\documentclass[preprint,12pt]{elsarticle}

\makeatletter
\def\ps@pprintTitle{%
 \let\@oddhead\@empty
 \let\@evenhead\@empty
 \def\@oddfoot{\centerline{\thepage}}%
 \let\@evenfoot\@oddfoot}
\makeatother
\usepackage{amsmath, amssymb, color, verbatim}
\usepackage[active]{srcltx}
\usepackage{amsthm}

\usepackage[utf8]{inputenc}
\usepackage[T1]{fontenc}

\author{Jonas Teuwen}%
\address{Delft Institute of Applied Mathematics,
  Delft University of Technology, P.O. Box 5031, 2600 GA Delft, The
  Netherlands}%

\usepackage{enumerate} 

\swapnumbers
\newtheorem{theorem}{Theorem}

\newtheorem{lemma}{Lemma}

\theoremstyle{remark}

\newcommand{\D}{\,\textup{d}}

\newcommand{\la}{\langle}
\newcommand{\ra}{\rangle}

\newcommand{\CcR}{{C_{\text{c}}(\R^d)}}

\renewcommand{\leq}{\leqslant}
\renewcommand{\geq}{\geqslant}
\renewcommand{\epsilon}{\varepsilon}

\newcommand{\R}{\mathbf R}
\newcommand{\e}{\mathrm{e}} 

\journal{Indagationes Mathematicae}

\begin{document}

\begin{frontmatter}



\title{A note on Gaussian maximal functions}




\begin{abstract}
  This note presents a proof that the non-tangential maximal function of the
  Ornstein-Uhlenbeck semigroup is bounded pointwise by the Gaussian
  Hardy-Littlewood maximal function. In particular this entails an extension on
  a result by Pineda and Urbina \cite{Pineda2008} who proved a similar result
  for a `truncated' version with fixed parameters of the non-tangential maximal
function.
  We actually obtain boundedness of the maximal function on non-tangential 
  cones of arbitrary aperture.
\end{abstract}


\end{frontmatter}


\section{Introduction}
Maximal functions are among the most studied objects in harmonic
analysis. 
It is well-known that the classical non-tangential maximal function associated
with the heat semigroup is bounded pointwise by the Hardy-Littlewood maximal
function, for every $x \in \R^d$, i.e.,
\begin{equation}\label{eq:classical}
  \sup_{\substack{(y, t) \in   \R^{d + 1}_+\\ |x - y| < t}} |\e^{t^2 \Delta}
u(y)| \lesssim \sup_{r
    > 0}  \frac1{|B_r(x)|}\int_{B_r(x)} |u| \D\lambda,
\end{equation}
for all locally integrable functions $u$ on $\R^d$ where $\lambda$ is
the Lebesgue measure on $\R^d$ (cf.\ \cite[Proposition II
2.1.]{Stein1993}).
Here the action of \emph{heat semigroup} $\e^{t \Delta} u = \rho_t \ast u$ is
given by a convolution of $u$ with the \emph{heat kernel}
\begin{equation*}
  \rho_t(\xi) := \frac{\e^{-|\xi|^2/4t}}{(4\pi t)^{\frac{d}2}}, \:\text{with}\:
  t > 0 \:\text{and}\: \xi \in \R^d.
\end{equation*}
In this note we are interested in its Gaussian counterpart. The change from
Lebesgue measure to the \textit{Gaussian measure}
\begin{equation}
  \label{eq:Gaussian-measure}
  \mathrm{d}\gamma(x) := \pi^{-\frac{d}2} \e^{-|x|^2} \D\lambda(x)
\end{equation}
introduces quite some intricate technical and conceptual difficulties which are
due to its non-doubling nature. Instead of the Laplacian, we will use its
Gaussian
analogue, the \emph{Ornstein-Uhlenbeck operator} $L$ which is given by
\begin{equation}
  \label{eq:Ornstein-Uhlenbeck-operator}
  L := \frac12 \Delta - \la x, \nabla \ra = -\frac12 \nabla^* \nabla,
\end{equation}
where $\nabla^*$ denotes the adjoint of $\nabla$ with respect to the measure
$\D\gamma$.
Our main result, to be proved in Theorem~\ref{thm:Gaussian-maximal-function},
is the following Gaussian analogue of \eqref{eq:classical}:
\begin{equation}
  \label{eq:main}
  \sup_{(y, t) \in \Gamma_x^{(A, a)}} |\e^{t^2 L} u(y)| \lesssim \sup_{r > 0}
  \frac1{\gamma(B_r(x))}\int_{B_r(x)} |u| \, \D\gamma.
\end{equation}
Here, 
\begin{equation}
  \label{eq:Gaussian-cone}
  \Gamma_x^{(A, a)} := \Gamma_x^{(A, a)}(\gamma) := \{(y, t) \in \R^{d + 1}_+
  \,: \, |x - y| < At \:\text{and}\: t \leq a m(x)\}
\end{equation}
is the \textit{Gaussian cone} with aperture $A$ and cut-off parameter $a$, and
\begin{equation}\label{eq:m-function}
  m(x) := \min\biggl\{1, \frac1{|x|} \biggr\}.
\end{equation}
As shown in \cite[Theorem 2.19]{Mattila1995} the centered Gaussian
Hardy-Littlewood maximal function is of weak-type $(1, 1)$ and is
$L^p(\gamma)$-bounded for $1 < p \leq \infty$. In fact, the same result holds
when the Gaussian measure $\gamma$ is replaced by any Radon measure $\mu$.
Furthermore, if $\mu$ is doubling, then these results even hold for the
\textit{uncentered} Hardy-Littlewood maximal function. For the Gaussian measure
$\gamma$ the uncentered weak-type $(1, 1)$ result is known to fail for $d > 1$ \cite{Sjogren1983}.
Nevertheless, the uncentered Hardy-Littlewood maximal function for $\gamma$ is
$L^p$-bounded for $1 < p \leq \infty$ \cite{Liliana2002}.

A slightly weaker version of the inequality \eqref{eq:main} has been proved by 
Pineda and Urbina \cite{Pineda2008} who showed that 
\begin{equation*}
  \sup_{(y, t) \in \widetilde{\Gamma}_x} |\e^{t^2 L} u(y)|
  \lesssim \sup_{r > 0}  \frac1{\gamma(B_r(x))}\int_{B_r(x)} |u| \D\gamma,
\end{equation*}
where
\begin{equation*}
  \widetilde{\Gamma}_x = \{(y, t) \in \R^d_+ : |x - y| < t \leq
  \widetilde{m}(x)\}
\end{equation*}
is the `reduced' Gaussian cone corresponding to the function
\begin{equation*}
  \widetilde{m}(x) = \min\biggl\{\frac12, \frac1{|x|}\biggr\}.
\end{equation*}
Our proof of \eqref{eq:main} is shorter than the one presented in
\cite{Pineda2008}. It has the further advantage of allowing
the extension to
cones with arbitrary aperture $A > 0$ and cut-off
parameter $a > 0$ without any additional technicalities. This additional
generality is important and has already been used by Portal (cf. the claim
made in \cite[discussion preceding Lemma 2.3]{Portal2014}) to prove the
$H^1$-boundedness of the Riesz transform associated with $L$.

\section{The Mehler kernel}
The \textit{Mehler kernel} (see e.g., \cite{Sjogren1997}) is the Schwartz
kernel associated to the Ornstein-Uhlenbeck semigroup $(\e^{tL})_{t \geq 0}$,
that is,
\begin{equation}
  \label{eq:Ornstein-Uhlenbeck-semigroup-integral}
  \e^{tL} u(x) = \int_{\R^d} M_t(x, \cdot) u \, \D\gamma.
\end{equation}
There is an abundance of literature on the Mehler kernel and its
properties. We shall only use the fact, proved e.g. in the survey paper
\cite{Sjogren1997}, that it is given explicitly by
\begin{equation}
  \label{eq:Mehler-kernel-Sjogren}
  M_t(x,y) = \frac{\exp\biggl(-\dfrac{|\e^{-t} x - y|^2}{1 - \e^{-2t}}
    \biggr)}{(1 - \e^{-2t})^{\frac{d}2}} \e^{|y|^2}.
\end{equation}
Note that the symmetry of the semigroup $\e^{tL}$ allows us to conclude
that $M_t(x, y)$ is symmetric in $x$ and $y$ as well. A formula for
\eqref{eq:Mehler-kernel-Sjogren} honoring this observation is:
\begin{equation}
  \label{eq:Mehler-kernel}
  M_t(x, y) = \frac{\exp\biggl(-\e^{-2t} \dfrac{|x - y|^2}{1
      - \e^{-2 t}}  \biggr)}{(1 - \e^{-t})^{\frac{d}2}}
  \frac{\exp\biggl(2\e^{-t} \dfrac{\la x, y \ra}{1 + \e^{-t}}
    \biggr)}{(1 + \e^{-t})^{\frac{d}2}}.
\end{equation}

\section{Some lemmata}
We use $m$ as defined in \eqref{eq:m-function} in our next lemma,
which is taken from \cite[Lemma 2.3]{MaasNeervenPortal2011}.
\begin{lemma}\label{lem:m-xy-equivalence}
  Let $a, A$ be strictly positive real numbers and $t > 0$. We have
  for $x, y \in \R^d$ that:
  \begin{enumerate}
  \item If $|x - y| < A t$ and $t \leq a m(x)$, then $t
    \leq a(1 + aA) m(y)$,
  \item If $|x - y| < A m(x)$, then $m(x) \leq (1 +
    A) m(y)$ and $m(y) \leq 2 (1 + A) m(x)$. 
  \end{enumerate}
\end{lemma}

The next lemma, taken from \cite[Proposition 2.1(i)]{Mauceri2007}, will come
useful when we want to cancel exponential
growth in one variable with exponential decay in the other as long
both variables are in a Gaussian cone.
For the reader's convenience, we include a short proof.
\begin{lemma}\label{lem:Cone-Gaussians-comparable}
  Let $\alpha > 0$ and $|x - y| \leq \alpha m(x)$. Then:
  \begin{equation*}
    \e^{-\alpha^2-2\alpha} \e^{|y|^2}
    \leq \e^{|x|^2} \leq
    \e^{\alpha^2(1 + \alpha)^2+2\alpha(1 + \alpha)} \e^{|y|^2} .
  \end{equation*}
\end{lemma}
\begin{proof}
  By the triangle inequality and $m(x)|x| \leq 1$ we get, 
  \begin{equation*}
    |y|^2 \leq (\alpha m(x) + |x|)^2 \leq \alpha^2 + 2 \alpha + |x|^2.
  \end{equation*}
  This gives the first inequality.  For the second we use
  Lemma~\ref{lem:m-xy-equivalence} to infer $m(x) \leq (1 + \alpha)
  m(y)$. Proceeding as before we obtain
  \begin{equation*}
    |x|^2 \leq \alpha^2 (1 + \alpha)^2 + 2 \alpha (1 + \alpha) + |y|^2,
  \end{equation*}
    which finishes the proof.
\end{proof}

\subsection{An estimate on Gaussian balls}
Let $B := B_t(x)$ be the open Euclidean ball with radius $t$ and center $x$
and let $\gamma$ be the Gaussian measure as defined by
\eqref{eq:Gaussian-measure}. We shall denote by $S_d$ the surface area 
of the unit sphere in $\R^d$.

\begin{lemma}\label{lem:Gaussian-ball-shift-lemma}
  For all $x \in \R^d$ and $t > 0$ we have the inequality:
  \begin{equation}\label{eq:Gaussian-ball-shift-lemma}
    \gamma(B_t(x)) \leq \frac{S_d}{\pi^{\frac{d}2}} \frac{t^d}d \e^{2 t|x|}
\e^{-|x|^2}.
  \end{equation}
\end{lemma}
\begin{proof}
  Remark that, with $B := B_t(x)$,
  \begin{align*}
    \int_B \e^{-|\xi|^2} \D\xi &= \e^{-|x|^2} \int_{B} \e^{-|\xi -
      x|^2} \e^{-2 \la x, \xi - x \ra} \D\xi\\
    &\leq \e^{-|x|^2} \int_{B} \e^{-|\xi - x|^2} \e^{2 |x| |\xi - x|}
    \D\xi\\
    &\leq \e^{-|x|^2} \e^{2 t|x|} \int_{B} \e^{-|\xi - x|^2} \D\xi\\
    &= \pi^{\frac{d}2} \e^{2 t|x|} \e^{-|x|^2} \gamma(B_t(0)).
  \end{align*}
  So, there holds that
  \begin{equation}\label{eq:Gaussian-ball-shift-lemma-proof-1}
    \gamma(B_t(x)) \leq \e^{2 t|x|} \e^{-|x|^2} \gamma(B_t(0)).
  \end{equation}
  We proceed by noting that
  \begin{equation*}
    \gamma(B_t(0)) \leq \pi^{-\frac{d}2} |B_t(0)| \leq \pi^{-\frac{d}2} t^d
\frac{S_d}d,
  \end{equation*}
  and combine this with the previous calculation to obtain
  \begin{equation*}
    \gamma(B_t(x)) \leq \frac{S_d}{\pi^{\frac{d}2}} \frac{t^d}d \e^{2 t|x|}
\e^{-|x|^2}.
  \end{equation*}
  This completes the proof.
\end{proof}

\subsection{Off-diagonal kernel estimates on annuli}
As is common in harmonic analysis, we often wish to decompose
$\R^d$ into sets on which certain phenomena are easier to handle. Here
we will decompose the space into disjoint annuli. 

Throughout this subsection we fix $x \in \R^d$, constants $A, a \geq 1$, and a
pair
$(y,t) \in \Gamma_x^{(A, a)}$. We use the notation $rB$ to mean the ball
obtained from the ball $B$ by multiplying its radius by $r$.

The annuli $C_k := C_k(B_t(y))$ are given by:
\begin{equation}
  \label{eq:C_k-annulus-decomposition}
  C_k :=
  \begin{cases}
   2B_t(y), &k = 0,\\
   2^{k + 1} B_t(y) \setminus 2^k B_t(y), &k \geq 1.
  \end{cases}
\end{equation}
So, whenever $\xi$ is in $C_k$, we get for $k
\geq 1$ that
\begin{equation}
  \label{eq:C_k-annulus-decomposition-expand}
  2^k t \leq |y - \xi| < 2^{k + 1} t.
\end{equation}
On $C_k$ we have the following bound for $M_{t^2}(y,\cdot)$:
\begin{lemma}\label{lem:On-diagonal-kernel-estimates-on-Ck}
  For all $\xi \in C_k$ for $k \geq 1$ we have:
  \begin{equation}
    \label{eq:On-diagonal-kernel-estimates-on-Ck}
    M_{t^2}(y, \xi) \leq \frac{\e^{|y|^2}}{(1 - \e^{-2t^2})^{\frac{d}2}}
    \exp\bigl(2^{k +  1} t |y| \bigr) \exp\Big(-\frac{4^k}{2 \e^{2 t^2}} \Bigr),
  \end{equation}
\end{lemma}
\begin{proof}
  Considering the first exponential which occurs in the Mehler kernel
  \eqref{eq:Mehler-kernel} together with
  \eqref{eq:C_k-annulus-decomposition-expand} gives for $k \geq 1$:
  \begin{align*}
    \exp\biggl(-\e^{-2t^2} \frac{|y - \xi|^2}{1 - \e^{-2t^2}} \biggr)
    &\overset{\phantom{(\dagger)}}{\leq} \exp\biggl(-\frac{4^k}{\e^{2t^2}}
    \frac{t^2}{1 - \e^{-2t^2}} \biggr)\\
    &\overset{(\dagger)}{\leq} \exp\biggl(-\frac{4^k}{2 \e^{2t^2}} \biggr),
  \end{align*}
  where $(\dagger)$ follows from $1 - \e^{-s} \leq s$ for $s \geq 0$. Using the
estimate $1 + s \geq 2s$ for $0 \leq s \leq 1$, we find for the second
exponential in the Mehler kernel \eqref{eq:Mehler-kernel}, by
\eqref{eq:C_k-annulus-decomposition-expand} that
  \begin{align*}
    \exp\biggl(2\e^{-t^2} \frac{\la y, \xi \ra}{1 + \e^{- t^2}} \biggr)
    & \leq \exp(|\la y, \xi \ra|)\\
    & \leq \exp(|\langle y, \xi-y\rangle|) \e^{|y|^2}\\
    & \leq \exp\bigl(2^{k + 1} t |y| \bigr) \e^{|y|^2}.
  \end{align*}
  Combining these estimates we obtain
\eqref{eq:On-diagonal-kernel-estimates-on-Ck}, as required.
\end{proof}

\section{The main result}
In this section we will prove our main theorem as mentioned in \eqref{eq:main}
for which the necessary preparations have already been made.
\begin{theorem}\label{thm:Gaussian-maximal-function}
  Let $A, a > 0$. For all $x \in \R^d$ and all $u \in \CcR$ we have
  \begin{equation}
    \label{eq:Maximal-function-cone}
    \sup_{(y, t) \in \Gamma_x^{(A, a)}} |\e^{t^2 L} u(y)| \lesssim
    \sup_{r > 0} \frac1{\gamma(B_r(x))}\int_{B_r(x)} |u| \, \D\gamma,
  \end{equation}
  where the implicit constant only depends on $A, a$ and $d$.
\end{theorem}
\begin{proof}
  We fix $x \in \R^d$ and $ (y, t) \in \Gamma_x^{(A, a)}$. The proof of
  \eqref{eq:Maximal-function-cone} is based on splitting the
  integration domain into the annuli $C_k$ as defined by
  \eqref{eq:C_k-annulus-decomposition} and estimating on each annulus. More
  explicit,
  \begin{equation}
    \label{eq:Maximal-function-cone-intermediate-step-1}
    |\e^{t^2 L} u(y)| \leq \sum_{k = 0}^\infty I_k(y),
    \:\text{where}\: I_k(y) := \int_{C_k} M_{t^2}(y, \cdot) |u(\cdot)|
    \,\D\gamma.
  \end{equation} 
  We have $t \leq a m(x) \leq a$ and, by Lemma~\ref{lem:m-xy-equivalence}, $t
  |y| \leq a(1 + aA)$. Together with
  Lemma~\ref{lem:On-diagonal-kernel-estimates-on-Ck} we infer, for $\xi \in
  C_k$ and $k \geq 1$, that
  \begin{align*}
    \label{eq:Mehler-kernel-estimate-one-sided-bound-1}
    M_{t^2}(y, \xi) &\leq \frac{\e^{|y|^2}}{(1 - \e^{-2t^2})^{\frac{d}2}}
    \exp(2^{k + 1} a(1 + aA)) \exp\Big(-\frac{4^k}{2 \e^{2 a^2}} \Bigr)\\
    &=: \frac{\e^{|y|^2}}{(1 - \e^{-2t^2})^{\frac{d}2}} c_k.
  \end{align*}
  Combining this with Lemma~\ref{lem:Cone-Gaussians-comparable}, we obtain
  \begin{equation}
    \label{eq:Mehler-kernel-estimate-one-sided-bound-1}
    M_{t^2}(y, \xi) \lesssim_{A, a} \frac{\e^{|x|^2}}{(1 -
\e^{-2t^2})^{\frac{d}2}} c_k.
  \end{equation}       
  Also, by \eqref{eq:C_k-annulus-decomposition-expand} we get
  \begin{equation*}
    |x - \xi| \leq |x - y| + |\xi - y| \leq (2^{k + 1} + A) t .
  \end{equation*}
  Let $K$ be the smallest integer such that $2^{k + 1} \geq A$ whenever $k \geq
  K$. Then it follows that $C_k$ for $k \geq K$ is contained in $B_{2^{k +
      2}t}(x)$ and for $k < K$ is contained in $B_{2At}(x)$. We set
  \begin{equation*}
    D_k := D_k(x) =
    \begin{cases}
      B_{2^{k + 2}t}(x) &\text{if $k \geq K$,}\\
      B_{2At}(x) &\text{elsewhere.}
    \end{cases}
  \end{equation*}
  Let us denote the supremum on right-hand side of
  \eqref{eq:Maximal-function-cone} by $M_\gamma u (x)$. Using
  \eqref{eq:Mehler-kernel-estimate-one-sided-bound-1}, we can bound the
  integral on the right-hand side of
  \eqref{eq:Maximal-function-cone-intermediate-step-1} by 
  \begin{align*}
    \int_{C_k}  M_{t^2}(y, \cdot) |u(\cdot)| \,\D\gamma & \lesssim_{A, a}
    c_k \frac{\e^{|x|^2}}{(1 - \e^{-2t^2})^{\frac{d}2}}   \int_{C_k}
    |u| \,\D\gamma\\ 
    &\leq c_k \frac {\e^{|x|^2}} {(1 -
      \e^{-2t^2})^{\frac{d}2}} \int_{D_k} |u| \,\D\gamma\\ 
    &\leq c_k \frac{\e^{|x|^2}}{(1 - \e^{-2t^2})^{\frac{d}2}} \gamma(D_k)
M_\gamma u(x),
  \end{align*}
  where we pause for a moment to compute a suitable bound for $\gamma(D_k)$. As
  above we have both $t|x| \leq a m(x)|x| \leq a$ and $t \leq a$. Together with
  Lemma~\ref{lem:Gaussian-ball-shift-lemma} applied to $D_k$ for $k \geq K$ we
  obtain:
  \begin{align*}
    \gamma(D_k) \e^{|x|^2} &\lesssim_A C^d \frac{S_d}{d} t^d 2^{kd} \e^{2^{k +
3} t |x|}
    \e^{-|x|^2} \e^{|x|^2}\\
    &\lesssim_{A, a, d} t^d 2^{k d} \e^{2^{k + 3} a}.
  \end{align*}
  Similarly, for $k < K$:
  \begin{equation*}
    \gamma(D_k) \e^{|x|^2} \lesssim_{A, a, d} t^d \e^{2 A a}.
  \end{equation*}
  Using the bound $t\leq a$, we can infer that
 \begin{equation*}
    \frac{t^d}{(1 - \e^{-2t^2})^{\frac{d}2}} \leq \frac{a^d}{(1 -
      \e^{-2a^2})^{\frac{d}2}} \lesssim_{a, d} 1.
  \end{equation*}
  (note that $s/(1-\e^{-s})$ is increasing). Combining these computations with
 the ones above for $k \geq K$ we get
  \begin{equation*}
    \int_{C_k}  M_{t^2}(y, \cdot) |u(\cdot)| \,\D\gamma \lesssim_{A,
      a, d} c_k 2^{k d} \e^{2^{k + 2} a} M_\gamma u(x),
  \end{equation*}
  while for $k < K$ we get
  \begin{equation*}
    \int_{C_k}  M_{t^2}(y, \cdot) |u(\cdot)| \,\D\gamma \lesssim_{A,
      a, d} c_k M_\gamma u(x).
  \end{equation*} 
 Similarly, for $\xi \in 2B_t(x)$ we obtain:
   \begin{equation*}
    I_0 := \int_{2B_t}  M_{t^2}(y, \cdot) |u(\cdot)| \,\D\gamma \lesssim_{A,
      a, d} M_\gamma u(x).
  \end{equation*}
  Inserting the dependency of $c_k$ upon $k$ as coming from
  \eqref{eq:Mehler-kernel-estimate-one-sided-bound-1}, we obtain the bound:
  \begin{align*}
    |\e^{t^2 L} u(y)| &= I_0 + \sum_{k = 1}^{K - 1} I_k + \sum_{k = K}^\infty
I_k\\
    &\lesssim_{A, a, d} \biggl[1 + \sum_{k =  1}^{K - 1} c_k +  \sum_{k = 
K}^\infty c_k
      2^{k d} \e^{2^{k + 2}a} \biggr] M_\gamma u(x),\\
    &\lesssim_{A, a, d} \biggl[1 + \sum_{k =  1}^{K - 1} \e^{-\frac{4^k}{2 \e^{2
a^2}}} + \sum_{k = K}^\infty 2^{k d} \e^{2^{k + 1} (1 + 2a + aA)}
\e^{-\frac{4^k}{2 \e^{2 a^2}}} \biggr] M_\gamma u(x),
  \end{align*}
  valid for all $(y, t) \in \Gamma_x^{(A, a)}$. As the sum on the right-hand
side
  evidently converges, we see that taking the supremum proves
  \eqref{eq:Maximal-function-cone}.
\end{proof}

\section{Acknowledgments}
This work initiated as part of a larger project in collaboration with Mikko
Kemppainen.
I would like to thank the referee for his/her useful suggestions.

\bibliographystyle{elsarticle-num}
\bibliography{Teuwen-GaussianMF}

\begin{thebibliography}{1}
\expandafter\ifx\csname url\endcsname\relax
  \def\url#1{\texttt{#1}}\fi
\expandafter\ifx\csname urlprefix\endcsname\relax\def\urlprefix{URL }\fi
\expandafter\ifx\csname href\endcsname\relax
  \def\href#1#2{#2} \def\path#1{#1}\fi

\bibitem{Pineda2008}
E.~Pineda, W.~R. Urbina,
  \href{http://www.emis.ams.org/journals/DM/v16-1/art7.pdf}{{Non Tangential
  Convergence for the Ornstein-Uhlenbeck Semigroup}}, Divulgaciones
  Matem\'{a}ticas 13~(2) (2008) 1--19.
\newline\urlprefix\url{http://www.emis.ams.org/journals/DM/v16-1/art7.pdf}

\bibitem{Stein1993}
E.~M. Stein, Harmonic analysis: real-variable methods, orthogonality, and
  oscillatory integrals, Vol.~43 of Princeton Mathematical Series, Princeton
  University Press, Princeton, NJ, 1993, with the assistance of Timothy S.
  Murphy, Monographs in Harmonic Analysis, III.

\bibitem{Mattila1995}
P.~Mattila, {Geometry of Sets and Measures in Euclidean Spaces}, Cambridge
  University Press, Cambridge, 1995.
\newblock \href {http://dx.doi.org/10.1017/CBO9780511623813}
  {\path{doi:10.1017/CBO9780511623813}}.

\bibitem{Sjogren1983}
P.~Sj\"{o}gren, {A Remark on the Maximal Function for Measures in
  $\mathbf{R}^n$}, American Journal of Mathematics 105~(5) (1983) 1231--1233.
\newblock \href {http://dx.doi.org/10.2307/2374340}
  {\path{doi:10.2307/2374340}}.

\bibitem{Liliana2002}
L.~Forzani, R.~Scotto, P.~Sj\"{o}gren, W.~Urbina, {On the $L^p$ boundedness of
  the non-centered Gaussian Hardy-Littlewood maximal function}, Proceedings of
  the American Mathematical Society 130~(1) (2002) 73--79.
\newblock \href {http://dx.doi.org/10.1090/S0002-9939-01-06156-1}
  {\path{doi:10.1090/S0002-9939-01-06156-1}}.

\bibitem{Portal2014}
P.~Portal, {Maximal and quadratic Gaussian Hardy spaces}, Revista
  Matem\'{a}tica Iberoamericana 30~(1)  79--108.

\bibitem{Sjogren1997}
P.~Sj\"{o}gren, \href{http://link.springer.com/10.1007/BF02656487}{{Operators
  associated with the Hermite semigroup -- a survey}}, The Journal of Fourier
  Analysis and Applications 3~(S1) (1997) 813--823.
\newblock \href {http://dx.doi.org/10.1007/BF02656487}
  {\path{doi:10.1007/BF02656487}}.
\newline\urlprefix\url{http://link.springer.com/10.1007/BF02656487}

\bibitem{MaasNeervenPortal2011}
J.~Maas, J.~van Neerven, P.~Portal, {Whitney coverings and the tent spaces
  $T^{1,q}(\gamma)$ for the Gaussian measure}, Arkiv f\"{o}r Matematik 50~(2)
  (2011) 379--395.
\newblock \href {http://dx.doi.org/10.1007/s11512-010-0143-z}
  {\path{doi:10.1007/s11512-010-0143-z}}.

\bibitem{Mauceri2007}
G.~Mauceri, S.~Meda,
  \href{http://linkinghub.elsevier.com/retrieve/pii/S0022123607002613}{{BMO and
  $H^1$ for the Ornstein–Uhlenbeck operator}}, Journal of Functional Analysis
  252~(1) (2007) 278--313.
\newblock \href {http://dx.doi.org/10.1016/j.jfa.2007.06.017}
  {\path{doi:10.1016/j.jfa.2007.06.017}}.
\newline\urlprefix\url{http://linkinghub.elsevier.com/retrieve/pii/S0022123607002613}

\end{thebibliography}

\end{document}